\documentclass[final]{siamltex}

\usepackage{amsfonts}

\usepackage{hyperref}

\usepackage{ifthen}
\usepackage{amssymb,amsmath}

\usepackage{graphicx}

\usepackage{algorithm}
\usepackage{algorithmic}

\newboolean{arxiv}
\setboolean{arxiv}{false}
\ifthenelse{\boolean{arxiv}}
{
\newcommand{\diag}{\mathrm{diag}}
\usepackage{amsthm}

\theoremstyle{plain}
 \newtheorem{theorem}{Theorem}
 \newtheorem{corollary}{Corollary}
 \newtheorem{proposition}{Proposition}
 \newtheorem{lemma}{Lemma}
 
 \theoremstyle{definition}
 \newtheorem{definition}{Definition}
}
{}
\newtheorem{remark}{Remark}
\newtheorem{counterex}{Counter-Example}

\ifthenelse{\boolean{arxiv}}
{
\newenvironment{proofeq}{\begin{proof}}{\qedhere \] \end{proof}}
}
{
\newenvironment{proofeq}{{\em Proof}. }{\qquad \endproof \]}
}
\newcommand{\norm}[1]{\lVert #1 \rVert}
\newcommand{\hilbert}{\mathrm{d}}

\newcommand{\abs}[1]{\lvert #1 \rvert}

\newcommand{\midd}{\mathrm{mid}}

\title{Convergence of Tomlin's HOTS algorithm}

\ifthenelse{\boolean{arxiv}}
{
\author{Olivier Fercoq}
\address{INRIA Saclay and CMAP Ecole Polytechnique, ({\tt olivier.fercoq@inria.fr}). The doctoral work of the author is supported by Orange Labs through the research contract CRE~3795 with INRIA}
}
{
\author{Olivier Fercoq \thanks{INRIA Saclay and CMAP Ecole Polytechnique, ({\tt olivier.fercoq@inria.fr}). The doctoral work of the author is supported by Orange Labs through the research contract CRE~3795 with INRIA}}
}

\begin{document}

\maketitle

%250 words max.

\begin{abstract}
The HOTS algorithm uses the hyperlink structure of the web to compute a vector of scores with which one can rank web pages. The HOTS vector is the vector of the exponentials of the dual variables of an optimal flow problem (the ``temperature'' of each page). The flow represents an optimal distribution of web surfers on the web graph in the sense of entropy maximization. 

In this paper, we prove the convergence of Tomlin's HOTS algorithm. We first study a simplified version of the algorithm, which is a fixed point scaling algorithm designed to solve the matrix balancing problem for nonnegative irreducible matrices. The proof of convergence is general (nonlinear Perron-Frobenius theory) and applies to a family of deformations of HOTS. Then, we address the effective HOTS algorithm, designed by Tomlin for the ranking of web pages. The model is a network entropy maximization problem generalizing matrix balancing. We show that, under mild assumptions, the HOTS algorithm converges with a linear convergence rate. The proof relies on a uniqueness property of the fixed point and on the existence of a Lyapunov function.

We also show that the coordinate descent algorithm can be used to find the ideal and effective HOTS vectors and we compare HOTS and coordinate descent on fragments of the web graph. Our numerical experiments suggest that the convergence rate of the HOTS algorithm may deteriorate when the size of the input increases. We thus give a normalized version of HOTS with an experimentally better convergence rate.
% Finally, we give an algorithm to compute the
% HOTS vector when bounds on the flow of websurfers are known.
\end{abstract}

%  \footnotetext[1]{INRIA Saclay and CMAP Ecole Polytechnique\\
%  {\tt\small olivier.fercoq@inria.fr}}% 
% \footnotetext[2]{INRIA Saclay and CMAP Ecole Polytechnique\\
%  {\tt\small stephane.gaubert@inria.fr}}%
%  \footnotetext[3]{The doctoral work of the author is supported by Orange Labs through the research contract CRE~3795 with INRIA}%

\section{Introduction}

Internet search engines use a variety of algorithms to sort web pages based on their text
content or on the hyperlink structure of the web.
In this paper, we focus on an algorithm proposed by Tomlin in~\cite{Tomlin-HOTS} 
for the ranking of web pages, called HOTS. 
It may also be used for other purposes like the ranking of sport teams~\cite{Govan-ranking}.
Like PageRank~\cite{Brin-Anatomy}, HITS~\cite{Kleinberg-HITS} and SALSA~\cite{Lempel-Salsa}, HOTS uses the hyperlink structure of the web (see also~\cite{LangvilleMeyer-SurveyEigenvectorMethods, LanMey-Beyond} for surveys on 
link-based ranking algorithms).
This structure is summarized in the web graph,
which is a digraph with a node for each web page and an arc between pages~$i$ 
and~$j$ if there is a~hyperlink from page~$i$ to page~$j$. 

The HOTS vector, used to rank web pages, is the vector of the exponentials of the dual variables 
of an optimal flow problem. The flow
represents an optimal distribution of web surfers on the web graph in the sense of
entropy maximization. The dual variable, one by page, is interpreted as the ``temperature'' of
the page, the hotter a page the better. In the case of the PageRank,
the flow of websurfers is determined by the uniform transition probability of following
one hyperlink in the current page. This transition rule is in fact arbitrary.
The HOTS model assumes that the web surfers choose the hyperlink to follow
by maximizing the entropy of the flow. 
Tomlin showed that this vector is solution of a nonlinear fixed point equation.
He then proposed a scaling algorithm to compute the HOTS vector, based on this fixed point equation. 

This algorithm solves the matrix balancing problem
studied among others in~\cite{Hartfiel-balancing, EavesSchneider, Schneider-dss, Schneider-scaling}. 
Given a $n \times n$ nonnegative matrix $A$, the matrix balancing problem
consists in finding a matrix $X$ of the form $X=D^{-1} A D$
with $D$ diagonal definite positive and such that $\sum_k X_{i,k}=\sum_j X_{j,i}$ for all $i$.
We shall compare Tomlin's HOTS algorithm with Schneider and Zenios's coordinate descent DSS algorithm~\cite{Schneider-dss}.
The main difference between these algorithms is that in coordinate descent, the scaling is done node
 by node in the network (i.e.\ in a Gauss-Seidel fashion) whereas in Tomlin's HOTS algorithm, the scaling is done all the nodes at the same time, in a Jacobi fashion.

A problem close to the matrix balancing problem is the equivalence scaling problem, where
given a $m \times n$ nonnegative matrix $A$, we search for a matrix $X$ of the
form $X=D_1 A D_2$ with $D_1$ and $D_2$ diagonal definite positive and
such that $X$ is bistochastic.
The Sinkhorn-Knopp~\cite{SinkhornKnopp} algorithm is a famous algorithm designed for the resolution of the scaling problem.
We may see HOTS algorithm as the analog of Sinkhorn-Knopp algorithm for the matrix balancing problem:
both algorithms correspond to fixed point iterations on the diagonal scalings.
Moreover, Smith~\cite{smith-sinkhorn} and Knight~\cite{Knight-Sinkhorn} proposed to
rank web pages according to the inverse of the corresponding entry in
the diagonal scaling.

 However, whereas Sinkhorn-Knopp algorithm~\cite{SinkhornKnopp} 
and the coordinate descent algorithm~\cite{Tseng-descentCoordinate} have been proved to converge, 
it does not seem that a theoretical result on the convergence of Tomlin's HOTS algorithm
has been stated in previous works, although experimentations~\cite{Tomlin-HOTS}
suggest that it is the case.
Indeed, Knight~\cite[Sec.~5]{Knight-Sinkhorn} rose the fact that Tomlin did not
state any convergence result for HOTS algorithm.
Another algorithm for the matrix balancing problem is given in~\cite{Johnson-DomEig},
based on the equivalence between the matrix balancing problem and the problem of minimizing
the dominant eigenvalue of an essentially nonnegative matrix under trace-preserving diagonal perturbations~\cite{Johnson-DomEigTheory}.

In this paper, we prove the convergence of Tomlin's HOTS algorithm.
We first study a simplified version of the algorithm that we call the ideal HOTS algorithm.
It is a fixed point scaling algorithm that solves the matrix balancing problem
for nonnegative irreducible matrices. We prove its convergence thanks to nonlinear
Perron-Frobenius theory (Theorem~\ref{thm:convidealHots}). 
The proof methods are general and apply to a family of deformations of HOTS.
Then, we address the effective HOTS algorithm, for the general case, which is the version designed by Tomlin for the ranking of web pages. Indeed the web graph is not strongly connected, which implies that the balanced matrix does not necessarily exist. The model is a nonlinear network entropy maximization problem which generalizes matrix balancing. 
We show in Theorem~\ref{thm:convHots} that under mild assumptions the HOTS algorithm converges with a linear rate of convergence. The proof relies on the properties of the ideal HOTS algorithm: 
uniqueness of the fixed point up to an additive constant
and decrease of a Lyapunov function at every step (Theorem~\ref{thm:thetadecroit}).

We also show that Schneider and Zenios's coordinate descent algorithm can be adapted to 
find the ideal and effective HOTS vectors.
We compare the HOTS algorithm and coordinate descent 
on fragments of the web graph in Section~\ref{sec:compalgos}. 
We considered small, medium and large size problems. In all cases the respective computational costs of both algorithms
were similar.
As the performances of the HOTS algorithm depends on the primitivity
of the adjacency matrix considered and coordinate descent does not, coordinate descent 
can be thought to have a wider range of applications.
However, the actual implementation of the HOTS algorithm is attractive for web scale problems:
whereas coordinate descent DSS uses at each iteration (corresponding to a given web page) information
from incoming and outgoing hyperlinks, the HOTS algorithm reduces to elementwise operations and left and right matrix vector products. Hence, an iteration of the HOTS algorithm can be performed without computing neither
 storing the transpose of the adjacency matrix.

We give an exact coordinate descent algorithm for
the truncated scaling problem defined in~\cite{Schneider-scaling1}
and we extend its use to the problem of computing the HOTS vector when some
bounds on the web surfers flow are known. Experimental
results show that exact coordinate descent is an efficient algorithm for web scale problems and 
that it is faster than the inexact coordinate descent algorithm presented in~\cite{Schneider-scaling}. 
Finally, we remarked that the convergence rate of the effective HOTS
algorithm seems to deteriorate when the size of the graph considered increases.
In order to overcome this feature, we propose a normalized version of 
the HOTS algorithm where we maximize a relative entropy of the flow
of web surfers instead of the classical entropy. A byproduct is that the associated ranking
favorizes pages with no outlink less than Tomlin's HOTS.

The paper is organized as follows. In Section~\ref{sec:nonlinPerronFrobeniusTheory}, 
we recall the main theorems of nonlinear Perron-Frobenius theory,
in Section~\ref{sec:idealhots}, we prove the convergence of the ideal HOTS algorithm
and we give a Lyapunov function for this algorithms. In Section~\ref{sec:effhots},
we give the convergence rate of the effective HOTS algorithm. In Section~\ref{sec:trunc},
we study the HOTS problem with bounds on the flow of web surfers.
In Section~\ref{sec:compalgos}, we compare various candidate algorithms to compute the HOTS vector
and in Section~\ref{sec:nomHots}, we give the normalized HOTS algorithm.

\section{Nonlinear Perron-Frobenius theory}
\label{sec:nonlinPerronFrobeniusTheory}

The classical Perron-Frobenius theorem (see~\cite{BermanPlemmons-nonnegMat} for instance) states 
that the spectral radius of a nonnegative matrix $A$ is an eigenvalue
(called the Perron root) and that there exists an associated eigenvector with nonnegative coordinates.
If, in addition, $A$ is irreducible, then the Perron root is simple and the (unique up to a multiplicative constant)
nonnegative eigenvector, called the Perron vector, has only positive entries.
The nonlinear Perron-Frobenius theory is an extension of the Perron-Frobenius theorem
to monotone and homogeneous maps.
It has a multiplicative and an additive formulation.

\begin{definition}
A map $T: \mathbb{R}_+^n \to \mathbb{R}_+^n$ is monotone if for all vectors $p$, $q$ such that $p \leq q$,
$T(p) \leq T(q)$.
A map $T: \mathbb{R}_+^n \to \mathbb{R}_+^n$ is homogeneous if for all vector $p$ and for all nonnegative real $\lambda$,
$T(\lambda p) = \lambda T(p)$.
\end{definition}

\begin{definition} \label{defn:monoHomA}
A map $T: \mathbb{R}^n \to \mathbb{R}^n$ is additively homogeneous if for all vector $p$ and for all real $\lambda$,
$T(\lambda + p) = \lambda + T(p)$.
\end{definition}

We can transform a multiplicative monotone, homogeneous map $T^\times$
 into a monotone, additively homogeneous map $T^+$ and
vice versa by the following operation called the ``logarithmic glasses'':
\[
 T^+(p)=\log(T^\times(\exp(p)))
\]
where $\log$ and $\exp$ act elementwise.

The following results show that monotone and nonexpansive maps are indeed nonexpansive.
Hence, they are well suited for iterative algorithms.

\begin{proposition}[\cite{CrandallTartar}]
An additively homogeneous map is nonexpansive for the sup-norm
if and only if it is monotone.
\end{proposition}

For a more general result, we shall need Hilbert's projective metric.

\begin{definition}
For $x,y$ two vectors of $\mathbb{R}^n$, Hilbert's projective metric between $x$ and $y$ is defined as
\vspace{-2ex}
\[
 \hilbert(x,y)=\log(\max_{i,j \in [n]} \frac{x_i y_j}{y_i x_j})
\]\vspace{-1ex}
\end{definition}

\begin{proposition}[\cite{Bushell-hilbertcontraction}]
Any monotone and homogeneous map is nonexpansive for Hilbert's metric.
\end{proposition}

\begin{definition}
For a map $T:\mathbb{R}^n \to \mathbb{R}^n$ or $T:\mathbb{R}_+^n \to \mathbb{R}_+^n$, we call the graph of $T$ and we denote it $G(T)$, the directed graph with nodes $1, \ldots, n$ and
an arc from $i$ to $j$ if and only if $\lim_{t \to +\infty} T(t e_i)= +\infty$ where $e_i$ is
the $i$th basis vector.
\end{definition}

The following results give conditions for the existence and
uniqueness of the ``eigenvector'' of a monotone, (additively or multiplicatively) homogeneous map.

\begin{theorem}[Theorem 2 in \cite{GaubGun-existenceEig}] \label{thm:GaubGun}
Let $T$ be a monotone, additively homogeneous map.
If $G(T)$ is strongly connected, then there exists $u\in \mathbb{R}^n$ and
$\lambda \in \mathbb{R}$ such that $T(u)=\lambda + u$.
We say that $u$ is an additive eigenvector of $f$.
\end{theorem}

% \begin{theorem}[Corollary 2.5 in \cite{Nussbaum-iteratedMaps}, Theorem 2.3 in \cite{Friedland2011}, Theorem 6.8 in \cite{GaubNuss-uniqueness}] \label{thm:NussbaumMultiplicative}
% Let $T:\mathbb{R}_+^n \to \mathbb{R}_+^n$ be a continuously differentiable 
% map which is monotone and homogeneous and has an eigenvector $p$ with only positive entries.
%  If $\nabla T(p)$ is irreducible, then there is at most one eigenvector in $(\mathbb{R}_+ \setminus \{0\})^n$
% up to a multiplicative factor.
%  If $\nabla T(p)$ is primitive, then all the orbits defined by 
% \[
%  p_{k+1} =  \frac{T(p_k)}{\norm{T(p_k)}}
% \]
% for a given norm $\norm{\cdot}$ converge to $\frac{p}{\norm{p}}$ linearly at a rate equal to 
% $\frac{\abs{\lambda_2}}{\rho}$, the ratio of the second and largest
% eigenvalues of $\nabla T$.
% \end{theorem}

\begin{theorem}[Corollary 2.5 in \cite{Nussbaum-iteratedMaps}, Theorem 2.3 in \cite{Friedland2011}, Theorem 6.8 in \cite{GaubNuss-uniqueness}] \label{thm:Nussbaum}
Let $T:\mathbb{R}^n \to \mathbb{R}^n$ be a continuously differentiable 
map which is monotone and additively homogeneous
and has an additive eigenvector $p \in \mathbb{R}^n$.
 If $\nabla T(p)$ is irreducible, then the eigenvector is unique up to an additive factor.
 If $\nabla T(p)$ is primitive, then all the orbits defined by 
\vspace{-1ex}
\[
 p_{k+1} =  T(p_k) - \psi(T(p_k))
\]
for a given additively homogeneous function $\psi: \mathbb{R}^n \to \mathbb{R}$ converge to $p- \psi(p)$ 
linearly at a rate equal to 
$\abs{\lambda_2(\nabla T(p))} = \max \{ \abs{\lambda} ; \lambda \in \mathrm{spectrum}(\nabla T(p)), \lambda \not = 1\}$.
\end{theorem}

These theorems have been stated with more general assumptions, among others
 semi-differentiability~\cite{GaubNuss-uniqueness} and
infinite state space~\cite{Nussbaum-iteratedMaps}. However, for the sake of simplicity,
we present them here in this simpler form. We can also write them theorem in the multiplicative form.

These results give a general framework to prove that the HOTS score and more generally
many web rankings are well defined,
i.e.\ that the score is unique and that the fixed point algorithm (or power algorithm) 
used to compute them indeed converges to the expected ranking.

\section{The ideal HOTS algorithm}
\label{sec:idealhots}

The web graph is a graph constructed from the hyperlink structure of the web.
Each web page is represented by a node and there is an arc between nodes
$i$ and $j$ if and only if page~$i$ points to page~$j$.
We shall denote by $A$ the adjacency matrix of the web graph.

There are two versions of the HOTS algorithm: an ideal version
for strongly connected graphs, i.e.\ for irreducible adjacency matrices,
and an effective version for general graphs that we will study in Section~\ref{sec:effhots}.
The HOTS algorithm for irreducible matrices is designed for the resolution of the following nonlinear network flow problem.
The optimization variable $\rho_{i,j}$ represents the traffic of websurfers on
the hyperlink from page~$i$ to page~$j$. 
\begin{align*}
\max_{\rho \geq 0} & -\sum_{ i,j \in [n]} \rho_{i,j} (\log(\frac{\rho_{i,j}}{A_{i,j}})-1)  \\
&\sum_{j \in [n]} \rho_{i,j} = \sum_{j \in [n]} \rho_{j,i} \; , \; \forall i \in [n]  & (p_i)\\
&\sum_{i,j \in [n]} \rho_{ij} =  1        & (\mu)
\end{align*}
The dual problem consists in minimizing the function $\theta$ on $\mathbb{R}^n \times \mathbb{R}$ where
\begin{align*}
\theta(p,\mu) :=  \sum_{i,j \in [n]} A_{ij} e^{p_i - p_j+\mu}-\mu\enspace .
\end{align*}
We use the convention that $0 \log(0) = 0$ and that $x \log(x/0)= 0$
if $x=0$ and $x \log(x/0)= +\infty$ otherwise.

If $(p,\mu)$ is a minimizer of $\theta$, then the value of $\exp(p_i)$ is interpreted as the temperature of page~$i$, the hotter the better.
We call it the HOTS (Hyperlinked Object Temperature Scale) score.

The ideal HOTS algorithm (Algorithm~\ref{alg:hots}) reduces to the fixed point iterations
for the function $f$ defined by
\begin{equation} \label{eqn:dssFixPoint}
 f(x)=\frac{1}{2} (\log(A^T e^x)-\log(A e^{-x}) ) \enspace.
\end{equation}
Denoting $y_i=e^{p_i}$, we can write it in multiplicative form to spare computing the exponentials and logarithms.
\begin{algorithm}
\caption{Ideal HOTS algorithm~\cite{Tomlin-HOTS}}
\label{alg:hots}
Start with an initial point $y_0 \in \mathbb{R}^n$, $y_0>0$. Given $y^k$, compute $y^{k+1}$ such that
\[
 y^{k+1}_i = \left( \frac{\sum_{j \in [n]} A_{j,i} y^k_j}{\sum_{l \in [n]} A_{i,l}(y_l^k)^{-1}} \right)^{\frac{1}{2}} \enspace .
\]
\end{algorithm}

\begin{algorithm}
\caption{Coordinate descent DSS~\cite{Schneider-dss}}
\label{alg:coorDescentDSS}
Start with an initial point $y^0 \in \mathbb{R}^n$, $y_0>0$.
Given $y^k$, select a coordinate $i \in [n]$ and compute $y^{k+1}$ such that
\vspace{-2ex}
\begin{align*}
 y^{k+1}_i=&\left( \frac{\sum_{j \in [n]} A_{j,i} y^k_j}{\sum_{l \in [n]} A_{i,l} (y_l^k)^{-1}} \right)^{\frac{1}{2}} \\
 y^{k+1}_j=&y^k_j \; , \qquad \forall j \not = i
\end{align*}
\end{algorithm}

We shall compare the HOTS algorithm with Schneider and Zenios's coordinate descent DSS algorithm (Algorithm~\ref{alg:coorDescentDSS}).
This is indeed a coordinate descent algorithm since for every $k$, we have, denoting $p_i=\log(y_i)$, 
\[p^{k+1}_i=\arg\min_{x \in \mathbb{R}} \theta(p^k_1, \ldots, p^k_{i-1}, x, p^k_{i+1}, \ldots, p^k_n) \enspace .\]
Coordinate descent algorithms (Algorithm~\ref{alg:coorDescent}) are designed to solve
\begin{align} \label{eq:coorDescMin}
 \min_{x \in \mathcal{X}} \phi(x)
\end{align}
where $\mathcal{X}$ is a possibly unbounded box 
of $\mathbb{R}^n$ and $\phi$ has the form
$\phi(x)=\psi(E x)+\langle b, x \rangle$,
$\psi$ is a proper closed convex function, $E$ is a $m \times n$ matrix
having no zero row and $b$ is a vector of $\mathbb{R}^n$.
\begin{algorithm}
\caption{Coordinate descent}
\label{alg:coorDescent}
Start with an initial point $x^0 \in \mathbb{R}^n$.
Given $x^k$, select a coordinate $i \in [n]$ and compute $x^{k+1}$ such that
\begin{align*}
 x^{k+1}_i=& \arg\min_{l_i \leq y \leq u_i} \phi(x^k_1, \ldots, x^k_{i-1}, y, x^k_{i+1}, \ldots, x^k_n)\\
 x^{k+1}_j=&x^k_j \; , \qquad \forall j \not = i
\end{align*}
\end{algorithm}
\begin{proposition}[\cite{Tseng-descentCoordinate}] \label{prop:coorDescent}
 Assume that the set of optimal solutions $\mathcal{X}^*$ of \eqref{eq:coorDescMin} is nonempty,
that the domain of $\psi$ is open, that $\psi$ is twice continuously differentiable on its domain
and that $\nabla^2 \psi(Ex)$ is positive definite for all $x \in \mathcal{X}^*$.
Let $(x^k)_k$ be a sequence generated by the coordinate descent algorithm (Algorithm~\ref{alg:coorDescent}),
using the cyclic rule (more general rules are also possible). Then $(x^k)_k$ 
converges at least linearly to an element of $\mathcal{X}^*$.
\end{proposition}

We now study the fixed point operator $f$ defined in \eqref{eqn:dssFixPoint}. 

\begin{proposition}[\cite{Tomlin-HOTS}] \label{prop:opt<>fix}
 A vector $p \in \mathbb{R}^n$ is a fixed point of $f$ defined in \eqref{eqn:dssFixPoint} if and only if the couple
$(p,\mu)$ with $\mu=-\log(\sum_{i,j \in [n]}A_{ij} e^{p_i - p_j})$ is a minimizer of the dual function. Moreover, in this case, denoting 
$D=\diag(\exp(p))$, $e^{\mu} DA D^{-1}$ is a maximizer of the network flow problem.
\end{proposition}
\begin{proof}
As $\theta$ is convex and differentiable, a couple $(p,\mu)$ is a minimizer if and only if it cancels the gradient.
 $\frac{\partial \theta}{\partial \mu}(p,\mu) =  \sum_{i,j \in [n]} A_{ij} e^{p_i - p_j} e^{\mu}-1$, so we have the expression of the optimal $\mu$
as a function of $p$. To conclude, we remark that
\[
  \frac{\partial \theta}{\partial p_k}(p,\mu)=\left (-\sum_{i \in [n]}A_{i,k} e^{p_i-p_k} + \sum_{j\in [n]} A_{k,j} e^{p_k-p_j} \right )e^{\mu}=0
\]
is equivalent to $f(p)=p$.
To get back to the primal problem, we remark that the primal cost of $e^{\mu}DA D^{-1}$
is equal to the dual cost of $(p,\mu)$ and that it is an admissible circulation.
\end{proof}

% \begin{definition}
% A map $f: \mathbb{R}^n \to \mathbb{R}^n$ is monotone if for all vectors $p$, $q$ such that $p \leq q$,
% $f(p) \leq f(q)$.
% A map $f: \mathbb{R}^n \to \mathbb{R}^n$ is additively homogeneous if for all vector $p$ and for all real $\lambda$,
% $f(\lambda + p) = \lambda + f(p)$.
% \end{definition}

\begin{proposition} \label{prop:M-AH}
 The map $f$ defined in \eqref{eqn:dssFixPoint} is monotone, additively homogeneous (Definition~\ref{defn:monoHomA}).
\end{proposition}
\begin{proof}
For all real $\lambda$ and for all vectors $p$, $q$ such that $p \leq q$, 
$f(\lambda + p) = \lambda+ f(p)$ ($\log(e^\lambda)=\lambda$) and $f(p) \leq f(q)$ ($\log$ and $\exp$ are increasing functions).
\end{proof}

The following result gives the conditions for the existence and uniqueness of the ideal HOTS vector.
\begin{theorem}[\cite{EavesSchneider}] \label{thm:Schneider}
 There exists $v \in \mathbb{R}^n$ such that $f(v)=v$ and $\sum_{i \in [n]} v_i=0$
if and only if $A$ has a diagonal similarity scaling if and only if $A$ is completely reducible.

If in addition $A$ is irreducible, then this vector is unique.
\end{theorem}

\begin{corollary}[\cite{Schneider-dss}] \label{prop:Schneider-dss}
 If $A$ is completely reducible, coordinate descent DSS (Algorithm~\ref{alg:coorDescentDSS})
converges linearly to a vector $v$ such that $\diag(v) A \diag(v)^{-1}$ is scaled.
\end{corollary}

To prove the convergence of the ideal HOTS algorithm (Algorithm~\ref{alg:hots}), we
use the nonlinear Perron-Frobenius theory, the main theorems of which are stated in Section~\ref{sec:nonlinPerronFrobeniusTheory}.

\begin{theorem} \label{thm:convidealHots}
Let $f$ be the map defined in \eqref{eqn:dssFixPoint}. If $A$ is irreducible and $A+A^T$ is primitive, then there exists a vector $v$ and such that $f(v)=v$  and for all $x \in \mathbb{R}^n$,
\[
 \limsup_{k \to \infty} \norm{f^{k+1}(x)-v}^{1/k} \leq \abs{\lambda_2(P)} = \max \{ \abs{\lambda} ; \lambda \in \mathrm{spectrum}(P), \lambda \not = 1\}
\]
where $P=\frac{1}{2}\left (\diag(A^T e^{v})^{-1} A^T \diag(e^v) + \diag(A e^{-v})^{-1} A \diag(e^{-v}) \right )$.
In particular, the ideal HOTS algorithm (Algorithm~\ref{alg:hots}) converges linearly at rate $\abs{\lambda_2(P)}$.
\end{theorem}
\begin{proof}
The iterates of the fixed point iteration defined by $p^0=x$ and $p^{k+1}=f(p^k)$ 
verify $p^k=\log(y^k)$ where $y^k$ is the $k$th iterate of the ideal HOTS algorithm (Algorithm~\ref{alg:hots})
started with $y^0=\exp(x)$. Hence, by continuous differentiability of the exponential, 
the rate of convergence of both versions of the algorithm is the same. 
By Theorem~\ref{thm:Schneider}, as $A$ is irreducible, $f$ has a fixed point $v$ and $\diag(\exp(v))$ is solution of the matrix balancing problem associated to $A$.
Now easy calculations show that $\nabla f =P$. As $P$ has the same pattern as $A+A^T$, $P$ is primitive if and only if $A+A^T$ is.
The result follows from Theorem~\ref{thm:Nussbaum}.
\end{proof}

This theorem shows that the HOTS vector for the irreducible case is
well defined if $A$ is irreducible and
that if $A+A^T$ is primitive, then the ideal HOTS algorithm~\eqref{eqn:dssFixPoint} converges linearly to the HOTS vector.

\begin{remark}
The ideal HOTS algorithm (Algorithm~\ref{alg:hots}) requires a primitivity assumption in order to converge
 that coordinate descent DSS (Algorithm~\ref{alg:coorDescentDSS}) does not require.
On the other hand, the convergence rate of coordinate descent DSS is not explicitly given
while Theorem~\ref{thm:convidealHots} gives the convergence rate of ideal HOTS.
\end{remark}

\begin{remark} \label{rem:diag}
Changing the diagonal of $A$
does not change the optimal scaling, so we can choose a nonzero diagonal for $A$ in the preceding theorem.
This is useful when $A$ is irreducible but not primitive.
\end{remark}

% \begin{proposition}
%  Suppose that $f$ has a fixed point. Then $A$ has a block diagonal structure with irreducible
% diagonal blocks.
% \end{proposition}
% \begin{proof}
% \end{proof}

The fixed point equation defining the ideal HOTS vector is
\[
 y_i = \left( \frac{\sum_j A_{j,i} y_j}{\sum_k A_{i,k} y_k^{-1}} \right)^{\frac{1}{2}} \enspace .
\]
Indeed, the page~$i$ has a good HOTS score if it is linked to by pages with a good HOTS score
and if it does not link to pages with a bad HOTS score.

We thus introduce the following set of fixed point ranking algorithms.
\begin{algorithm}
\caption{Deformed HOTS algorithm}
\label{alg:defHots}
Let $\alpha, \beta \geq0$ such that $\alpha+\beta=1$ and let $g:\mathbb{R}^n_{+} \to \mathbb{R}^n_{+}$ defined for all $i$ by
 \[
 g_i(x)=\frac{(\sum_j A_{j,i} x_j)^\alpha}{(\sum_k A_{i,k} x_k^{-1})^\beta}  \enspace .
\]
Given an initial point $d_0 \in \mathbb{R}^n$ and a norm $\norm{\cdot}$, the deformed HOTS algorithm is defined by
\vspace{-2ex}
\[
 d^{k+1}=\frac{g(d^k)}{\norm{g(d^k)}}
\]
\end{algorithm}

\begin{proposition}
Let $\alpha, \beta \geq 0$ such that $\alpha+\beta=1$. 
If $A$ is irreducible and $\alpha A+\beta A^T$ is primitive, then the deformed HOTS algorithm (Algorithm~\ref{alg:defHots})
converges linearly to a positive vector.
\end{proposition}
\begin{proof}
Let $h=\log \circ g \circ \exp$. As in the proof of Theorem~\ref{thm:convHots}, 
the rate of convergence for the fixed point iterations with $g$ or $h$ is the same.
The map $h$ is monotone and additively homogeneous. For $\alpha>0$, its graph
is equal to $A$. Hence, for $\alpha>0$, $h$ has an eigenvector by Theorem~\ref{thm:GaubGun}.
For $\alpha=0$, as $A$ is irreducible, by the Perron-Frobenius theorem~\cite{BermanPlemmons-nonnegMat}, 
$A$ has an eigenvector $x$. Then $\log(x^{-1})$ is an eigenvector of $h$.
Now,
$\nabla h(v)=\alpha \diag(A^T e^{v})^{-1} A^T \diag(e^v) + \beta \diag(A e^{-v})^{-1} A \diag(e^{-v})$,
so we have the convergence as soon as $\alpha A+\beta A^T$ is primitive
by Theorem~\ref{thm:Nussbaum}.
\end{proof}

\begin{remark} \label{rem:antiperron}
For $\alpha=\frac{1}{2}$, we have the fixed point diagonal similarity scaling, for $\alpha=1$,
we have the ranking by the Perron vector~\cite{Keener-FootbalTeams} and for $\alpha=0$, we have an
``anti-Perron'' score,
where good pages are those that do not link to pages with a bad score. 
\end{remark}

The following result gives a global contraction factor in
the case when $A$ is positive.
% \begin{definition}
% For $x,y$ two vectors of $\mathbb{R}^n$, Hilbert's projective metric between $x$ and $y$ is defined as
% \[
%  \hilbert(x,y)=\max_{i,j \in [n]} x_i-y_i+y_j-x_j
% \]
% \end{definition}

\begin{proposition} \label{prop:contract}
If $k(A)$ is the contraction factor of $A$ in Hilbert metric ($k(A)<1$ if $A$ is positive), 
then $f$ is $\frac{k(A^T)+k(A)}{2}$-contracting in Hilbert metric. 
\end{proposition}

\begin{proofeq}
 Let $x$ and $y$ be two positive vectors such that $\eta y \leq x \leq \nu y$ elementwise.
Then $\eta' A^T y \leq A^T x \leq \nu' A^T y$ with $\log(\nu'/\eta')\leq k(A^T) \log(\nu/\eta)$.
We also have that $\eta'' A y^{-1} \leq A x^{-1} \leq \nu'' A y^{-1}$ with $\log((\eta'')^{-1}/(\nu'')^{-1})\leq k(A) \log(\nu/\eta)$.
Hence, 
\[d(g(x), g(y))=\log((\frac{\nu'' \nu'}{\eta'' \eta'})^{1/2}) \leq \frac{k(A^T)+k(A)}{2}\log(\frac{\nu}{\eta}) \enspace .
\end{proofeq}

A key technical ingredient of the convergence of the effective HOTS algorithm described in the next section will be Theorem~\ref{thm:thetadecroit} below
showing that each iteration $p \leftarrow f(p)$ of the ideal HOTS algorithm does not increase 
the dual objective function.
% \begin{proposition}
%  $\theta(\frac{f(p)+p}{2})\leq \theta(p)$
% \end{proposition}
% \begin{proof}
% Let us denote $\psi(p,q)=\sum_{i,j} e^{p_i} A_{ij} e^{-q_j}$. 
% \[
%  \psi(p,f(p))=\sum_{i,j} e^{p_i} A_{ij} \left( \frac{(Ae^{-p})_j}{(A^Te^{p})_j} \right)^{-1/2} = \sum_j (A^T e^p)_j^{1/2} (A e^{-p})_j^{1/2}
% \]
%  From the coordinate descent DSS algorithm, for all $j$, $(A^T e^p)_j^{1/2} (A e^{-p})_j^{1/2} \leq (A^T e^p)_j e^{-p_j}$,
% thus we have $\psi(p,f(p))=\psi(f(p),p) \leq \psi(p,p)=\theta(p)$.
% 
% Now As $\psi$ is convex, $\theta(\frac{f(p)+p}{2})=\psi(\frac{1}{2}(f(p),p)+\frac{1}{2}(p,f(p)))\leq \theta(p)$.
% \end{proof}

\begin{theorem}[Lyapunov function] \label{thm:thetadecroit}
 $\theta(f(p))\leq \theta(p)$
\end{theorem}

\begin{proofeq} Let us denote $\psi(p,q)=\sum_{i,j} e^{p_i} A_{ij} e^{-q_j}$. 
\begin{align*}
 \psi(p,2 f(p)-p)&=\sum_{i,j} e^{p_i} A_{ij} \left( \frac{(Ae^{-p})_j}{(A^Te^{p})_j} \right)e^{p_j} = \sum_j (A e^{-p})_j e^{p_j} \\
&=\theta(p)=\psi(2 f(p)-p,p)
\end{align*}

Now, as $\psi$ is convex, 
\[
\theta(f(p))=\psi(\frac{1}{2}(2f(p)-p,p)+\frac{1}{2}(p,2f(p)-p))\leq \theta(p) \enspace.
\end{proofeq}
\section{The effective HOTS algorithm} \label{sec:effhots}

Theorem~\ref{thm:Schneider} gives conditions for the existence and uniqueness of the HOTS vector in
the ideal case. In practice the irreducibility condition does not hold for the web graph.
The classical solution for this problem is to add a small positive value to the adjacency matrix~\cite{Brin-Anatomy, LanMey-Beyond}
in order to get a positive matrix. Tomlin proposed an alternative approach based on the network flow model.
We consider the following nonlinear network flow problem with network given by
\vspace{-2ex}
\[
A'= \begin{bmatrix}
  A & 1 \\ 1^T & 0
 \end{bmatrix}
\]
where $1$ denotes the vector with all entries equal to 1.
\begin{align*}
\max_{\rho \geq 0} & \; - \!\!\!\!\! \sum_{ i,j \in [n+1]} \rho_{i,j} (\log(\frac{\rho_{i,j}}{A'_{i,j}})-1)  \\
&\sum_{j \in [n+1]} \rho_{i,j} = \sum_{j \in [n+1]} \rho_{j,i} \; , \; \forall i \in [n+1]  & (p_i) \\
&\sum_{i,j \in [n+1]} \rho_{ij} =  1        & (\mu) \\
&\sum_{j \in [n]} \rho_{n+1,j} =  1-\alpha & (a) \\
&1-\alpha = \sum_{i \in [n]} \rho_{i,n+1}&  (b)
\end{align*}
We use the conventions that $0 \log(0)=0$ and that $x \log(x/0)=0$ if and only if $x=0$.
In this new model, we add an artificial node connected to all the other nodes and such that the flow through this node
is precribed to be $1-\alpha$.

The algorithm is designed for the minimization of the dual function $\theta$ where
\begin{multline} \label{eq:dualeff}
\theta(p,\mu,a,b) =  \sum_{i,j \in [n]} A_{ij} e^{p_i - p_j+\mu} + \sum_{i \in [n]} e^{-b-p_{n+1}+p_i+\mu} \\
+ \sum_{j \in [n]} e^{a+p_{n+1}-p_j +\mu} -(1-\alpha)a-\mu+(1-\alpha)b \enspace .
\end{multline}

We first give the following counter-example, showing that the problem may be ill posed.
\begin{counterex}\label{countex:effHots}
%The map $f^\alpha$ is not nonexpansive for Hilbert's metric. Moreover, successive applications of $f^\alpha$ may diverge.
The dual function $\theta$ may be unbounded.
\end{counterex}

\begin{proof}
 Take
\vspace{-2ex}
\[
 A=\begin{bmatrix} 0 & 1 & 0 \\ 0 & 0 & 1\\ 0 & 0 & 0 \end{bmatrix} .
\]
We have 
\vspace{-2ex}
\begin{align*}
 \tilde{\theta}(p)=C(\alpha) + (1-\alpha)\big (\log(\sum_{i\in [n]} e^{p_i})+\log(\sum_{i\in [n]} e^{-p_i}) \big)
+(2\alpha-1) \log(e^{p_1-p_2} + e^{p_2-p_3})
\end{align*}
where $C(\alpha)\in \mathbb{R}$. For all $k \in \mathbb{R}$,
\begin{align*}
\tilde{\theta}&(-k,0,k,0)=C(\alpha)+2(1-\alpha)\log(1+e^k+e^{-k}) + (2\alpha-1) \log(e^{-k} + e^{-k}) \\
  &= C(\alpha)+2(1-\alpha)k-(2\alpha-1)k+2(1-\alpha)\log(1+e^{-k}+e^{-2k})+(2\alpha-1)\log(2)
\end{align*}
For  $\alpha>\frac{3}{4}$, $\theta$ is unbounded. %, so that $f$ cannot have any fixed point.
\end{proof}

This example is indeed rather degenerate: the HOTS algorithm can only diverge
because it searches the minimum of an unbounded function. Said otherwise,
it tries to solve a network flow problem without any admissible flow.
We shall give conditions under which there exists a HOTS vector
and show that the HOTS algorithm converges to the HOTS vector when these conditions hold.

\begin{remark}
 A natural idea to establish the convergence of a fixed point algorithm
is to show that it is a contraction in Hilbert metric. In the case of the effective
HOTS algorithm, even when the matrix $A$ is positive, the fixed point algorithm
may not be a contraction (take a perturbation of Counter-example~\ref{countex:effHots}).
\end{remark}

\begin{lemma}[\cite{Tomlin-HOTS}] \label{lem:lambda}
For any $p\in \mathbb{R}^{n+1}$, the minimum of $\theta(p,\mu,a,b)$ with respect to $\mu$, $a$ and $b$ is unique and 
given by 
\begin{align*}
\mu&= \log(\frac{2\alpha-1}{\sum_{i,j \in [n]} A_{i,j}e^{p_i-p_j}}) \enspace, \\
a&=\log(\frac{1-\alpha}{2\alpha-1}\frac{\sum_{i,j\in [n]} A_{i,j}e^{p_i-p_j}}{\sum_{j\in[n]}e^{p_{n+1}-p_j}})\enspace, \\
b&=-\log(\frac{1-\alpha}{2\alpha-1}\frac{\sum_{i,j\in [n]} A_{i,j}e^{p_i-p_j}}{\sum_{i\in[n]}e^{p_i-p_{n+1}}})\enspace.
\end{align*}
\end{lemma}

\begin{proof}
 The function $\theta(p,\cdot,\cdot,\cdot)$ is convex and differentiable so the optimality condition
is just that the gradient is zero. One can easily see that the only triple that cancels the gradient is
the one given in the lemma.
\end{proof}

We denote $\lambda=(\mu,a,b)$ and $\lambda(p)$ the solution of the minimization of $\theta(p,\lambda)$ with respect to $\lambda$.
For $\lambda \in \mathbb{R}^3$, we denote
\begin{equation} \label{eqn:def-flambda}
 f^\lambda_i(p)=\frac{1}{2} (\log(\sum_{j\in [n]} A_{j,i} e^{p_j}+e^{p_{n+1}+a})-\log(\sum_{k\in [n]}A_{i,k} e^{-p_k} +e^{-p_{n+1}-b}) ) \enspace.
\end{equation}
We also define $g^{\lambda}=\exp \circ f^\lambda \circ \log$:
\begin{equation*}
 g^\lambda_i(y)=\left ( \frac{\sum_{j\in [n]} A_{j,i} y_j+e^{a}y_{n+1}}{\sum_{k\in [n]}A_{i,k} (y_k)^{-1} +e^{-b}(y_{n+1})^{-1}} \right)^{\frac{1}{2}} \enspace.
\end{equation*}
\begin{algorithm}
 \caption{Effective HOTS algorithm}
\label{alg:effhots}
Given an initial point $y^0 \in \mathbb{R}^n$, the effective HOTS algorithm is defined by
\[
 y^{k+1}=g^{\lambda(\log(y^k))}(y^k)
\]
\end{algorithm}

%The HOTS algorithm consists in successive applications of $p \leftarrow f^{\lambda(p)}(p)$.

\begin{lemma}[\cite{Fercoq-PerronOpt}] \label{lem:strictconvex}
Let $d \in \mathbb{R}^{n}$ such that for all $i$, $d_i>0$. The function $\tilde{\theta}$ defined by $\tilde{\theta}(p)=\min_{p_{n+1},\mu,a,b} \theta((p,p_{n+1}),\mu,a,b)$
is stricly convex on the hyperplane $H=\{ x \in \mathbb{R}^{n} \; |  \; \sum_{i \in [n]} d_i x_i = 0\}$.
In particular, there exists at most one HOTS vector up to an additive constant.
\end{lemma}

\begin{proof}
From the expressions of $a+p_{n+1}$, $b+p_{n+1}$ and $\mu$ at the optimum (Lemma~\ref{lem:lambda}), respectively given by $e^{a+p_{n+1}} = \frac{(1-\alpha)e^{-\mu}}{\sum_{j \in [n]}e^{-p_j}}$, $e^{-b-p_{n+1}} = \frac{(1-\alpha)e^{-\mu}}{\sum_{i \in [n]}e^{p_i}}$ and $e^{\mu}=\frac{2\alpha-1}{\sum_{i,j \in [n]} A_{ij} e^{p_i - p_j}}$,
we can write $\tilde{\theta}$ as
\begin{equation*} 
 \tilde{\theta}(p) =C(\alpha)+(1-\alpha) \phi(-p) +(1-\alpha) \phi(p) + (2\alpha -1) \log(\sum_{i,j\in [n]} A_{i,j} e^{p_i-p_j})
\end{equation*}
where $\phi: x\mapsto \log (\sum_{i \in [n]} e^{x_i})$ is the log-sum-exp function, which is strictly convex on any
hyperplane that does not containing the vector with all entries equal to 1,
and $C(\alpha)=1 -2(1-\alpha) \log(1-\alpha)-(2\alpha-1) \log(2\alpha-1)$.
$\tilde{\theta}$ is the sum of convex functions and stricly convex functions,
so it is stricly convex on $H$. 
We conclude that the minimum of $\tilde{\theta}$ on $H$ is unique if it exists.
We can then extend this result to the whole space since $\theta(\eta+p)=\theta(p)$ for all real number $\eta$.
\end{proof}

\begin{lemma} \label{lem:coerciv}
Let $d$ and $\tilde{\theta}$ be as in Lemma~\ref{lem:strictconvex}. The function $\tilde{\theta}$ is coercive on the hyperplane $\{ x \in \mathbb{R}^{n+1} \; |  \; \sum_{i \in [n+1]} d_i x_i = 0\}$ if and only if there
exists a primal solution with the same pattern as $A'$.
\end{lemma}

\begin{proof}
 If the function $\theta$ is coercive, there exists a dual solution and thus
there also exists a primal solution with the same pattern as $A'$.

If there exists a primal solution with the same pattern as $A'$, the constraint
qualification conditions are satisfied~\cite{Schneider-scaling1}, and there
exists a dual solution. By Lemma~\ref{lem:strictconvex}, $\tilde{\theta}$ is strictly convex on the hyperplane $\{ x \in \mathbb{R}^{n+1} \; |  \; \sum_{i \in [n+1]} d_i x_i = 0\}$. Thus it is
necessarily coercive on this hyperplane.
\end{proof}

% \begin{lemma} \label{lem:uniqueness}
% Up to an additive constant, there exists at most one HOTS vector.
% \end{lemma}
% \begin{proof}
% This is a consequence of Proposition 10 in~\cite{Fercoq-PerronOpt}, which shows in particular that the function $(p \mapsto \theta(p,\lambda(p))$ is
% stricly convex on the hyperplane $\{ x \in \mathbb{R}^{n+1} \; |  \; \sum_i x_i = 0\}$.
% \end{proof}

\begin{lemma} \label{lem:convLambdafixed}
If $A \not = 0$, then for any fixed $\lambda$, the iterative algorithm consisting in successive applications of the map $f^{\lambda}$, defined in \eqref{eqn:def-flambda}, converges
to a minimizer of the function $(p \mapsto \theta(p,\lambda))$. Moreover, this minimizer is unique up to an additive constant.
\end{lemma}

\begin{proof}
The map $f^{\lambda}$ corresponds to the ideal HOTS fixed point operator~\eqref{eqn:dssFixPoint} for
the matrix 
\[
 \begin{bmatrix}
  A & e^a 1 \\ e^b 1^T & 0
 \end{bmatrix}
\]
where $1$ denotes the vector with all entries equal to 1. As this matrix is primitive as soon as $A \not = 0$,
Theorem~\ref{thm:convidealHots} and Proposition~\ref{prop:opt<>fix} apply.
\end{proof}

\begin{theorem}
\label{thm:convrate}
Let $F$ defined by $F(p)=f^{\lambda(p)}(p)$ as in \eqref{eqn:def-flambda}.
Let $p^*$ be the logarithm of a HOTS vector defined by $p^*=F(p^*)$. The matrix
$\nabla F(p^*)$
has all its eigenvalues in the real interval $(-1,1]$ and the eigenvalue 1 is simple.
\end{theorem}

\begin{proofeq}
Let us denote $\gamma=\frac{1-\alpha}{2\alpha -1}$, $e^a=\gamma \frac{\sum_{i,j\in [n]} A_{ij} e^{p_i-p_j}}{\sum_j e^{-p_j}}$,
 $e^{-b}=\gamma \frac{\sum_{i,j \in [n]} A_{ij} e^{p_i-p_j}}{\sum_i e^{p_i}}$. Then for all $k \in [n]$,
 \begin{align*}
 F_k(p)&=\frac{1}{2} \log ( \sum_{i\in [n]} A_{ik}e^{p_i} + e^a ) - \frac{1}{2}\log ( \sum_{j \in [n]} A_{kj}e^{-p_j}+e^{-b} ) \\
F_{n+1}(p)&=\frac{1}{2} \log (\sum_{i\in [n]} e^{p_i}) - \frac{1}{2}\log ( \sum_{j \in [n]}e^{-p_j}) \enspace.
 \end{align*}
As no coordinate of $F(p)$ depends on $p_{n+1}$, we may consider the
reduced function that we shall still denote $F$ and such that to
$p \in \mathbb{R}^n$ associates $F(p, 0)$. The eigenvalues of the gradient of
original function are 0 and the eigenvalues of the gradient of the reduced function.

First, we have 
\[
F(p)=p+\frac{1}{2} \log(1- \diag(d) \frac{\partial \tilde{\theta}}{\partial p})
\]
where $\tilde{\theta}(p)=\min_{p_{n+1},\mu,a,b} \theta((p,p_{n+1}),\mu,a,b)$ and
$d_l=\frac{1}{2\alpha-1}\frac{e^{ p_l}(\sum_{i,j\in [n]} A_{ij} e^{p_i-p_j})}{\sum_{i \in [n]} A_{il}e^{p_i} + e^a } >0$
(see~\cite{Fercoq-PerronOpt} for more details on this calculus).
Differentiating this equality, we deduce that
$\nabla F= I_n -\frac{1}{2}  \diag(d) \nabla^2 \tilde{\theta}$.
Let $\lambda$ be an eigenvalue of $\nabla F$. This means that there exist a vector $x$ such that
\[
 \lambda x = \nabla F x= x -\frac{1}{2}  \diag(d) \nabla^2 \tilde{\theta} x
\]
\[
  \nabla^2 \tilde{\theta} x = 2(1-\lambda) \diag(d^{-1}) x
\]
This is a generalized eigenvalue problem with $\nabla^2 \tilde{\theta}$ symmetric semi-definite positive by convexity of $\tilde{\theta}$ and $\diag(d^{-1})$
diagonal definite positive. Hence $2(1-\lambda)$ is necessarily a nonnegative real number and
$\lambda$ is real and smaller than 1. Also, if $\lambda=1$, this means that $x$ is the
vector with all its entries equal to 1 (by Proposition~10 in~\cite{Fercoq-PerronOpt} which is a simple extension of
Lemma~\ref{lem:strictconvex}) and thus $\lambda$ is simple.

We shall now show that all the eigenvalues of $\nabla F(p^*)$ are stricly greater than $-1$.
Differentiating the expression of $a$, we get
\[
 \frac{\partial e^a}{\partial p_l}(p)=\gamma \left( -\frac{\sum_{i \in [n]} A_{il}e^{p_i-p_l}}{\sum_{j \in [n]} e^{-p_j}} +\frac{\sum_{j \in [n]} A_{lj}e^{p_l-p_j}}{\sum_{j \in [n]} e^{-p_j}} + \frac{\sum_{i,j \in [n]} A_{ij} e^{p_i-p_j}}{(\sum_{j \in [n]} e^{-p_j})^2}e^{-p_l} \right )
\]
But as $p^*$ is a fixed point of $F$, it satisfies the equality
\[
 \sum_{i \in [n]} A_{il}e^{p^*_i-p^*_l} + e^{a^*-p_l^*}  = \sum_{j \in [n]} A_{lj}e^{p^*_l-p^*_j}+e^{-b^*+p_l^*}
\]
which can be rewritten as
\[
 \sum_{j \in [n]} A_{lj}e^{p^*_l-p^*_j}- \sum_{i \in [n]} A_{il}e^{p^*_i-p^*_l}=\gamma \sum_{i,j \in [n]} A_{ij} e^{p^*_i-p^*_j} \left( \frac{e^{-p^*_l}}{\sum_{j \in [n]} e^{-p^*_j}} - \frac{e^{p^*_l}}{\sum_{i \in [n]} e^{p^*_i}} \right) \enspace .
\]
Hence
\[
  \frac{\partial e^a}{\partial p_l}(p^*)=\gamma \sum_{i,j \in [n]} A_{ij} e^{p^*_i-p^*_j} \left( \gamma \frac{e^{-p^*_l}}{(\sum_{j} e^{-p^*_j})^2} -\gamma \frac{e^{p^*_l}}{\sum_{i} e^{p^*_i}\sum_{j} e^{-p^*_j}} + \frac{e^{-p^*_l}}{(\sum_{j} e^{-p^*_j})^2} \right )
\]
Let us introduce $d'$ such that
\begin{equation} \label{eq:dprime}
 {d'}_k^{-1}= \sum_{i \in [n]} A_{ik}e^{p^*_i-p^*_k} + e^{a^*-p^*_k} =  \sum_{j \in [n]} A_{kj}e^{p^*_k-p^*_j}+e^{-b^*+p^*_k}\enspace.
\end{equation}
Doing the same for $e^{-b}$ as for $e^{a}$ and differentiating $F$, we get
\begin{multline*}
 \frac{\partial F_k}{\partial p_l}(p^*)=\frac{1}{2} d'_k A_{l,k} e^{p^*_l-p^*_k} + \frac{1}{2} d'_k A_{k,l}e^{p^*_k-p^*_l}
+\frac{1}{2} \gamma \sum_{i,j \in [n]} A_{ij} e^{p^*_i-p^*_j} d'_k \times \\
\left( (1+\gamma) \frac{e^{-p^*_l-p^*_k}}{(\sum_{j} e^{-p^*_j})^2} - \gamma \frac{e^{p^*_l-p^*_k}}{\sum_{i} e^{p^*_i}\sum_{j} e^{-p^*_j}} -\gamma  \frac{e^{p^*_k-p^*_l}}{\sum_{i} e^{p^*_i}\sum_{j} e^{-p^*_j}} +(1+\gamma) \frac{e^{p^*_l+p^*_k}}{(\sum_{i} e^{p^*_i})^2}\right).
\end{multline*}
We can now decompose $\frac{\partial F}{\partial p}$ as
\[
 \frac{\partial F}{\partial p}=D' S + D' R
\]
where $D' =\diag(d')$, $S$ is a symmetric matrix with nonnegative entries
\[
S_{k,l} =\frac{1}{2} A_{k,l} e^{p_l-p_k} + \frac{1}{2} A_{l,k}e^{p_k-p_l}
+\frac{1}{2} \gamma \sum_{i,j} A_{ij} e^{p_i-p_j} 
\left( \frac{e^{-p_l-p_k}}{(\sum_j e^{-p_j})^2} +\frac{e^{p_l+p_k}}{(\sum_i e^{p_i})^2}\right)
\]
 and $R$ is the following symmetric
rank 1 matrix 
\[
 R_{k,l}=\frac{1}{2} \gamma^2 \sum_{i,j} A_{ij} e^{p_i-p_j} \left(\frac{e^{-p_k}}{\sum_j e^{-p_j}} -\frac{e^{p_k}}{\sum_i e^{p_i}} \right) \left(\frac{e^{-p_l}}{\sum_j e^{-p_j}} -\frac{e^{p_l}}{\sum_i e^{p_i}}\right) \enspace .
\]
The nonnegative matrix $D' S$ verifies that for all $k$, $\sum_l d'_k S_{k,l}=1$, 
thus by the Perron-Frobenius theorem~\cite{BermanPlemmons-nonnegMat},
we have exhibited a Perron vector and the spectral radius of the matrix is 1. Moreover, 
$D' S$ is positive, so every other of its eigenvalues has a modulus strictly smaller than 1.

The matrix $D' R$ is a rank 1 matrix and its only nonzero eigenvalue is positive. Indeed it is equal to $\frac{1}{2} \gamma^2 \sum_{i,j} A_{ij} e^{p_i-p_j}\sum_k (\frac{e^{-p_k}}{\sum_j e^{-p_j}} -\frac{e^{p_k}}{\sum_i e^{p_i}})^2 d'_k $.

Let $\lambda$ be an eigenvalue of $ \frac{\partial F}{\partial p}=D' S + D' R$.
It then also an eigenvalue of the similar matrix $(D')^{1/2} S(D')^{1/2} + (D')^{1/2} R(D')^{1/2}$, which is symmetric.
Hence,
\[
 \lambda \geq \min_{x \in \mathbb{R}^n : \norm{x}_2=1} x^T (D')^{1/2} S(D')^{1/2} x + x^T (D')^{1/2} R(D')^{1/2} x
\]
 As the spectral radius of $D' S$ is 1, the same is true for $ (D')^{1/2} S(D')^{1/2}$ and for all vector~$x$, $x^T (D')^{1/2} S(D')^{1/2} x > -\norm{x}^2_2$.
As $D' R$ has only nonnegative eigenvalues, $ (D')^{1/2} R(D')^{1/2}$ is semi-definite positive and $x^T (D')^{1/2} R(D')^{1/2} x \geq 0$ for all $x$.
As a conclusion, 
\[ 
\lambda \geq \min_{x \in \mathbb{R}^n : \norm{x}_2=1} x^T (D')^{1/2} S(D')^{1/2} x + x^T (D')^{1/2} R(D')^{1/2} x > -1\enspace .
\end{proofeq}

\begin{theorem} \label{thm:convHots}
Let $F(p)=f^{\lambda(p)}(p)$ as in \eqref{eqn:def-flambda}. If there exists a primal feasible point with the same pattern as~$A'$,
 then the effective HOTS algorithm (Algorithm~\ref{alg:effhots}) converges to a HOTS vector $e^{p^*}$ (unique up to a multiplicative constant)
 linearly at a rate
$\abs{\lambda_2(\nabla F(p^*))}=\max \{ \abs{\lambda} ; \lambda \in \mathrm{spectrum}(\nabla F(p^*)), \lambda \not = 1\}$,
\end{theorem}

\begin{proof}
% First of all, the algorithm is well defined: the integer $l$ is
% finite since $\lim_l \theta((f^{\lambda_k})^l(p_k), \lambda_k) = \min_p \theta(p,\lambda_k)$
% by Lemma~\ref{lem:convLambdafixed}.
Let $\tilde{F}$ be the map defined by $\tilde{F}(p)=F(p)-\frac{1}{1^T {d'}^{-1}} \sum_{i \in [n+1]} (d'_i)^{-1} F_i(p)$,
with $d_i'$ defined in \eqref{eq:dprime} in the proof of Theorem~\ref{thm:convrate} 
for $i\in [n]$ and $(d'_{n+1})^{-1}=0$.
For all $k$, 
let $p_k$ be the $k$th iterate of the HOTS algorithm, i.e.\ $p_{k+1}=F(p_k)$, and let $\lambda_k=\lambda(p_k)$.
We also define $q_k$ by $q_0=p_0$ and $q_{k+1}=\tilde{F}(q_k)$.
By Theorem~\ref{thm:thetadecroit} and by definition of $\lambda_{k+1}$, we have 
$\theta(p_k, \lambda_k)\geq \theta(p_{k+1}, \lambda_k) \geq \theta(p_{k+1}, \lambda_{k+1})$.
As $q_k-p_k$ is proportional to the vector with all entries equal to 1,  $\lambda(p_k)=\lambda(q_k)=\lambda_k$
and $\theta(q_k, \lambda)=\theta(p_k, \lambda)$ for all $k$ and $\lambda$.
Hence 
\begin{equation} \label{eqn:thetadecroit}
\theta(q_k, \lambda_k)\geq \theta(q_{k+1}, \lambda_k) \geq \theta(q_{k+1}, \lambda_{k+1}) \enspace .
\end{equation}
Now, for all $k$, $q_k \in H=\{ x \in \mathbb{R}^{n+1} \; |  \; \sum_{i \in [n+1]} (d'_i)^{-1} x_i = 0\}$. 
As by Lemma~\ref{lem:coerciv}, $\tilde{\theta}$ is coercive on $H$, 
$\theta$ is bounded from below and $(\theta(q_k, \lambda_k))_k$ converges to, say, $\bar{\theta}$.
Moreover, the sequence $(q_k)_k$ must be bounded. By continuity of the function $\lambda(\cdot)$, $(\lambda_k)_k$ is also bounded. Hence, they have limit points.

Let $\bar{q}$ be a limit point of $(q_k)_k$ and $\bar{\lambda}=\lambda(\bar{q})$. For all $\epsilon>0$ and $K>0$, there exists $k\geq K$ and $k' \geq k+1$ such that
$\norm{q_k-\bar{q}}\leq \epsilon$, $\norm{q_{k'}-\bar{q}}\leq \epsilon$. 
By~\eqref{eqn:thetadecroit},
\[
 \theta(q_k, \lambda_k) \geq \theta(q_{k+1}, \lambda_k)\geq \theta(q_{k+1}, \lambda(q_{k+1})) \geq \theta(q_{k'}, \lambda_{k'})
\]
where $q_{k+1}=\tilde{F}(q_k)$.
When $\epsilon$ tends to 0 and $K$ tends to infinity, we get with $\bar{\bar{q}}=\tilde{F}(\bar{q})$,
\[
 \theta(\bar{q}, \bar{\lambda}) \geq \theta(\bar{\bar{q}}, \bar{\lambda})\geq \theta(\bar{\bar{q}}, \lambda(\bar{\bar{q}})) \geq \theta(\bar{q}, \bar{\lambda})\enspace .
\]
In particular, $\theta(\bar{\bar{q}}, \bar{\lambda})= \theta(\bar{\bar{q}}, \lambda(\bar{\bar{q}}))$. This implies by Lemma~\ref{lem:lambda}
that $\bar{\lambda} \in \arg\min_\lambda \theta(\bar{\bar{q}}, \lambda)$ and thus
$\bar{\lambda}=\lambda(\bar{\bar{q}})$ by uniqueness of the minimizer.

Similarly, $\bar{\bar{q}}$ is also a limit point of $(q_k)_k$ and we may consider the sequence $(u_k)_k$ such that
 $u_k=(\tilde{F})^k(\bar{q})=(f^{\bar{\lambda}})^k(\bar{q})-\frac{1}{1^T {d'}^{-1}}\sum_{i\in [n+1]} {d'_i}^{-1} ((f^{\bar{\lambda}})^k(\bar{q}))_i$.
% Iterating the argument of the preceeding paragraph, for any $K$, there exists $k \geq K$ such that $u_k$ is a limit point of $(p_m)_m$.
Iterating the argument of the preceeding paragraph, for any $k$, $\lambda(u_k)=\bar{\lambda}$ and $u_k$ is a limit point of $(q_k)_k$.
Now, by Lemma~\ref{lem:convLambdafixed}, the sequence $(u_k)$ converges to $q^* \in \arg\min_q \theta(q, \bar{\lambda})$.
As we also have $\lambda(q^*)=\bar{\lambda}$, we conclude that $(q^*, \bar{\lambda})$ is a minimizer of $\theta$ and that
there exists a limit point of $(q_k, \lambda_k)_k$ that minimizes $\theta$.
Now, as $(\theta(q_k, \lambda_k))_k$ is decreasing, all the limit points of $(q_k)$ minimize $\theta$.
The uniqueness of the minimizer of $\theta$ on $H$ (Lemma~\ref{lem:strictconvex})
gives the convergence of the effective HOTS algorithm to the HOTS vector in the projective space,
that is the convergence of the sequence $(q_k)_k$.

We shall now prove that the sequence $(p_k)_{k \geq 0}$ indeed converges.
Let us denote by $\rho=\max \{ \abs{\lambda} ; \lambda \in \mathrm{spectrum}(\nabla F(q^*)), \lambda \not = 1\}$.
By Theorem~\ref{thm:convrate}, we know $\rho<1$.
Denoting $(\lambda_i, u_i, v_i)_{i \in [n+1]}$, the eigenvalues and eigenvectors of $\nabla F(q^*)=\sum_{i = 1}^{n+1} \lambda_i u_i v_i^T$, we have 
$\nabla \tilde{F}(q^*) = \sum_{i = 2}^{n+1} \lambda_i u_i v_i^T + \frac{1 ({d'}^{-1})^T} {1^T {d'}^{-1}} - \frac{1}{1^T {d'}^{-1}} ({d'}^{-1})^T \frac{1( {d'}^{-1})^T} {1^T {d'}^{-1}}= \sum_{i = 2}^{n+1} \lambda_i u_i v_i^T$.
Hence by~\cite{Ostrowski}, for all $\epsilon' >0$, there exists a norm $\norm{\cdot}$ such that
for all $x \in \mathbb{R}^{n+1}$, $\norm{\nabla \tilde{F}(q^*)x} = \norm{\sum_{i = 2}^n \lambda_i u_i v_i^T x} \leq (\rho+\epsilon') \norm{x}$.

Thus for $x\in \mathbb{R}^{n+1}$ sufficiently close to $q^*$,
\[
 \norm{\tilde{F}(x)-q^*}=\norm{\tilde{F}(x)-\tilde{F}(q^*)}\leq (1+\epsilon'/2) \norm{\nabla {F}(q^*) (x-q^*)} \leq (\rho+\epsilon') \norm{x-q^*} \enspace .
\]
We deduce that $(q_k)$ converges linearly at rate $\rho$ to $q^*$.
Now for all $k$, we have 
\begin{align*}
p_k&=q_k+\sum_{l=0}^{k-1} \frac{1}{1^T {d'}^{-1}}\sum_{i\in [n+1]} {d'_i}^{-1} F_i(q_l) \\
&=q_k+\sum_{l=0}^{k-1} \frac{1}{1^T {d'}^{-1}}\sum_{i\in [n+1]} {d'_i}^{-1} F_i(q_l)-F_i(q^*) = q_k+\sum_{l=0}^{k-1} \eta_l \enspace ,
\end{align*}
where $\abs{\eta_l} = \text{O} (\norm{F(q_l)-F(q^*)}) = \text{O} (\norm{q_l-q^*}) = \text{O}(\rho^l)$.
Hence $\sum_{l=0}^{k-1} \eta_l$ is summable and converges linearly at rate $\rho$.
Finally, $(p_k)$ converges linearly at rate $\rho$.
Like in the proof of Theorem~\ref{thm:convHots} we deduce the convergence of
the sequence $(\exp(p_k))$ linearly at rate $\rho$ to a HOTS vector.
\end{proof}

% For the needs of the proof, we had to modify the HOTS algorithm. However in practice,
% we have never encountered a case where the integer $l$ is bigger than $1$, that
% is a case where the classical HOTS algorithm is different from the modified HOTS algorithm.
% Thus we conjecture that this is always true, and thus that the HOTS algorithm converges
% as soon as there exists a primal feasible point with the same pattern as $A$.

% \begin{remark}
% If instead of considering the following nonlinear network flow problem where the network is given by
% \[
% A'= \begin{bmatrix}
%   A & 1 \\ 1^T & 0
%  \end{bmatrix}
% \]
% we take
% \[
% A''= \begin{bmatrix}
%   A & 1 \\ 1^T & 1
%  \end{bmatrix}
% \]
% then there always exists a primal feasible point with the same pattern~as $A''$ and the modified HOTS algorithm always converges (even for $A=0$).
% All the proofs carry over, with only a minor change in the expression of $\lambda(p)$.
% \end{remark}

The last result shows that coordinate descent is an alternative algorithm
for the computation of the effective HOTS vector.

\begin{proposition} \label{prop:effHotsCoorDescent}
If there exists a primal feasible point with the same pattern as~$A'$, 
the coordinate descent algorithm applied to the unrestricted minimization
of the dual function $\theta$ defined in~\eqref{eq:dualeff}
and choosing coordinates in a cyclic order
converges linearly to a HOTS vector.
\end{proposition}

\begin{proof}
If there exists a primal feasible point with the same pattern as~$A'$,
then the set of minimizers of $\theta$ is nonempty by Lemmas~\ref{lem:strictconvex} and~\ref{lem:coerciv}.
The function $\theta$ has the required form with $\psi(x)=\sum_{i,j} A_{i,j} \exp(x_{i,j})$. 
The hessian of $\psi$ is clearly definite positive for all $x$.
Thus the hypotheses of Proposition~\ref{prop:coorDescent} are verified and the result follows.
\end{proof}

\section{An exact coordinate descent for the truncated scaling problem}
\label{sec:trunc}

Truncated scaling problems were introduced by Schneider in~\cite{Schneider-scaling1,Schneider-scaling}
in order to generalize both matrix balancing and row-column equivalence scaling.
Given a $n \times n$ matrix $A$ and bounds $L$ and $U$ such that $L_{i,j} \leq U_{i,j}$, the truncated scaling problem
consists in finding a matrix $X$ of the form $X=\mathrm{T}(D^{-1} A D)$
with $D$ diagonal definite positive, $\mathrm{T}$ the truncation operator $\mathrm{T}_{i,j}(X)=\max(\min(U_{i,j}, X_{i,j}),L_{i,j})$
and such that $\sum_k X_{i,k}=\sum_j X_{j,i}$ for all~$i$.
This problem is equivalent to the following optimization problem
\begin{align*}
\max_{\rho \geq 0} & \; - \!\!\!\!\! \sum_{ i,j \in [n]} \rho_{i,j} (\log(\frac{\rho_{i,j}}{A_{i,j}})-1)  \\
\sum_{j \in [n]} \rho_{i,j} = &\sum_{j \in [n]} \rho_{j,i} \; , \; \forall i \in [n]  & (p_i) \\
L_{i,j} \leq \rho_{i,j} &\leq U_{i,j} \; , \; \forall i,j \in [n] & (\eta_{i,j} , \zeta_{i,j})
\end{align*}
If the bounds satisfy $L_{i,j}=0$ and $U_{i,j}=+\infty$ for all $i$, $j$, then we have a matrix balancing problem.
The reduction of row-column equivalence scaling to truncated scaling lies in
a graph transformation described by Schneider in~\cite{Schneider-scaling1}, Lemma~1.

The dual function can take two forms. In~\cite{Schneider-scaling1}, Schneider proposes to relax
the equality constraints and to let the bound constraints in the objective function.
One gets a dual function of the form
\begin{equation}\label{eq:dualtruncSchneider}
 \Psi(p) =  \sum_{i,j \in [n]} \psi^*_{i,j}(p_i-p_j)
\end{equation}
where each $\psi^*_{i,j}: \mathbb{R} \to \mathbb{R}$ is convex.
Then one can perform an inexact coordinate descent where
 at each step the minimization along the coordinate is not necessarily exact.

Here, we shall study the choice of relaxing all the constraints.
This approach has been proposed in~\cite{Zenios-rowactionAlgoforTransportation} for
another generalization of the row-column equivalence scaling problem
(but this generalization does not include truncated scaling). Then, 
we get the following dual function:
\begin{align} \label{eq:dualtrunc}
\theta(p, \eta, \zeta) =  \sum_{i,j \in [n]} \phi^*_{i,j}( p_i-p_j +\eta_{i,j}-\zeta_{i,j}) - \sum_{i,j \in [n]} L_{i,j} \eta_{i,j} + \sum_{i,j \in [n]} U_{i,j} \zeta_{i,j} \enspace ,
\end{align}
where $\phi_{i,j}^*(t)=A_{i,j} e^t$.
We shall minimize $\theta$ with unrestricted $p$ and nonnegative $\eta$ and $\zeta$.
As in~\cite{Censor-intervalConstrainedBalancing, Zenios-rowactionAlgoforTransportation},
we shall show in Proposition~\ref{prop:truncation} that exact expressions of the minimizers along one single coordinate exist.

\begin{proposition} \label{prop:truncation}
Given $p \in \mathbb{R}^n$, the minimizers of $\min_{\eta \geq 0, \zeta \geq 0} \theta(p, \eta, \zeta)$
are given for all $i$ and $j$ by
\begin{align*}
 \exp(\eta_{i,j})&= \max(\frac{L_{i,j}}{A_{i,j}e^{p_i-p_j}},1) \\
 \exp(-\zeta_{i,j})&= \min(\frac{U_{i,j}}{A_{i,j}e^{p_i-p_j}},1) 
\end{align*}
\end{proposition}
\begin{proof}
 The proofs for $\eta$ and $\zeta$ are symmetric, so we only do the one for $\eta$.
\begin{align*}
 \frac{\partial \theta}{\partial \eta_{i,j}} = A_{i,j}e^{p_i-p_j+\eta_{i,j}-\zeta_{i,j}}-L_{i,j}
\end{align*}
Two cases may occur: either $\frac{\partial \theta}{\partial \eta_{i,j}} = 0$ and $\eta_{i,j}\geq0$ or $\eta_{i,j}=0$ and $\frac{\partial \theta}{\partial \eta_{i,j}}\geq 0$.
Hence, either $\exp(\eta_{i,j})= \frac{L_{i,j}}{A_{i,j}e^{p_i-p_j-\zeta_{i,j}}}$ and $\exp(\eta_{i,j})\geq 1$ or
we have $\exp(\eta_{i,j})=1$ and $1 \geq \frac{L_{i,j}}{A_{i,j}e^{p_i-p_j-\zeta_{i,j}}} \geq \frac{L_{i,j}}{A_{i,j}e^{p_i-p_j}}$.
We thus have the result if $\eta_{i,j}$ and $\zeta_{i,j}$ are not positive together.

Now suppose that $\eta_{i,j}>0$ and $\zeta_{i,j}>0$. In this case, we have $\exp(\eta_{i,j})= \frac{L_{i,j}}{A_{i,j}e^{p_i-p_j-\zeta_{i,j}}}$
and $\exp(-\zeta_{i,j})= \frac{U_{i,j}}{A_{i,j}e^{p_i-p_j+\eta_{i,j}}}$.
This implies that $U_{i,j}=L_{i,j}$. Thus the two bound constraints are
in fact an equality constraint and we shall consider the unconstrained multiplier 
$\xi_{i,j}=\eta_{i,j}-\zeta_{i,j}$ instead of the two former multipliers.
Then $\xi_{i,j}$ verifies $\exp(\xi_{i,j})= \frac{L_{i,j}}{A_{i,j}e^{p_i-p_j}}$.
It is unique and can be redecomposed into $\zeta_{i,j}$ and $\eta_{i,j}$ as in the proposition.
\end{proof}

This proposition shows that we do not need to store the values of $\eta_{i,j}$ and $\zeta_{i,j}$.
Indeed, for fixed $p$ and $\lambda$, if $(\eta,\zeta)=\arg\min_{\eta' \geq 0, \zeta' \geq 0} \theta(p,\lambda, \eta', \zeta')$,
then if we denote by $\midd(x,a,b)=\max(\min(x,a),b)$
 \begin{align*}
A_{i,k}e^{p_i-\eta_{i,k}+\zeta_{i,k}}&=\midd(A_{i,k}e^{p_i}, U_{i,k} e^{p_k}, L_{i,k} e^{p_k}) \\
A_{k,j}e^{-p_j-\eta_{k,j}+\zeta_{k,j}}&=\midd(A_{k,j}e^{-p_j}, U_{k,j} e^{-p_k}, L_{k,j} e^{-p_k}) \enspace.
 \end{align*}
This gives an expression of coordinate descent with no storage of $\eta_{k+1}$ nor $\zeta_{k+1}$.

\begin{algorithm}
 \caption{Exact coordinate descent for Truncated scaling}
\label{alg:trunc}
Given an initial vector $p^0$, calculate iteratively $p^k$ by selecting a coordinate $l$ following
a cyclic rule and computing: $p_{l'}^{k+1}=p_{l'}^k$ for all $l' \not = l$ and 
\begin{multline*}
 p_l^{k+1}=\frac{1}{2} \log\left(\sum_{i\in [n]} \midd(A_{i,l}e^{p_i^k}, U_{i,l} e^{p_l^k}, L_{i,l} e^{p_l^k})\right) \\
 -\frac{1}{2} \log\left(\sum_{j\in [n]}\midd(A_{l,j}e^{-p_j^k}, U_{l,j} e^{-p_l^k}, L_{l,j} e^{-p_l^k})\right)
\end{multline*}
\end{algorithm}

\begin{proposition} \label{prop:truncCoorDescent}
If $A$ has a truncated scaling, 
then Algorithm~\ref{alg:trunc}
converges linearly to a solution of the truncated scaling problem.
\end{proposition}
\begin{proof}
By Proposition~\ref{prop:truncation}, we can see that Algorithm~\ref{alg:trunc}
is a coordinate descent algorithm (Algorithm~\ref{alg:coorDescent})
applied to the minimization of the dual function $\theta$ defined in~\eqref{eq:dualtrunc}
such that for every $l$, the coordinate selection order is
\[
 \eta_{l,1}, \ldots, \eta_{l,n}, \zeta_{l,1}, \ldots, \zeta_{l,n}, p_l
\]
As in Proposition~\ref{prop:effHotsCoorDescent}, we then just
verify the hypotheses of Proposition~\ref{prop:coorDescent}.
\end{proof}

\begin{remark}
 Due to the truncation, it is not clear how to determine the primitivity
of the gradient of the fixed point operator of a HOTS-like algorithm
for the truncated scaling problem.
\end{remark}

Tomlin proposed in~\cite{Tomlin-HOTS} to search for a flow of websurfers $\rho$ that
maximizes the entropy for the effective HOTS problem (Section~\ref{sec:effhots})
with additional bound constraints of the type 
\begin{equation} \label{eqn:hotsbounds}
U_{i,j} \leq \rho_{i,j} \leq L_{i,j} \enspace.
\end{equation}
These bound constraints
model the fact that one may have information on the actual flow of websurfers
through some hyperlink, even if the flow on every hyperlink is out of reach.

We propose the following coordinate descent algorithm (Algorithm~\ref{alg:HotsBounds}) 
that couples the algorithms presented in
Proposition~\ref{prop:effHotsCoorDescent}
and Algorithm~\ref{alg:trunc}. The proof of convergence is just the concatenation of the 
arguments of Propositions~\ref{prop:effHotsCoorDescent} and~\ref{prop:truncCoorDescent}.

\begin{algorithm}
\caption{Coordinate descent for the HOTS problem with bounds}
\label{alg:HotsBounds}
Start with an initial point $y^0 \in \mathbb{R}^n$, $y^0>0$.
Given $y^k$, select a coordinate $i \in [n+1]$ and compute $y^{k+1}$ such that 
$y^{k+1}_j=y^k_j$, $\forall j \not = l$ and
\begin{align*}
 y^{k+1}_l= \left( \frac{ \sum_{i\in [n]} \midd(A_{i,l}y_i^k, U_{i,l} y^k, L_{i,l} y_l^k) + e^{a^k} y_{n+1}^k}{\sum_{j\in [n]}\midd(A_{l,j}(y_j^k)^{-1}, U_{l,j} (y_l^k)^{-1}, L_{l,j} (y_l^k)^{-1}) + e^{-b^k}(y_{n+1}^k)^{-1}} \right )^{\frac{1}{2}} 
\end{align*}
If $i<n+1$, then set $(\mu^{k+1}, a^{k+1}, b^{k+1})=(\mu^{k}, a^{k}, b^{k})$,
otherwise 
\begin{align*}
\mu^{k+1}&= \log(\frac{2\alpha-1}{\sum_{j,j' \in [n]} \midd(A_{j,j'}\frac{y_j^k}{y_{j'}^k}, U_{j,j'},  L_{j,j'} )}) \enspace, \\
a^{k+1}&=\log(\frac{1-\alpha}{\sum_{j\in[n]}\frac{y^k_{n+1}}{y^k_j}e^{\mu^{k+1}}})\enspace, \qquad
b^{k+1}=-\log(\frac{1-\alpha}{\sum_{j'\in[n]}\frac{y^k_{j'}}{y^k_{n+1}}e^{\mu^{k+1}}})\enspace.
\end{align*}
\end{algorithm}

% \begin{proposition}
% Algorithm~\ref{alg:trunc} launched on the truncated scaling problem for $(A,L,U)$ which corresponds to $(\tilde{A}, \alpha)$
% is equivalent to the Sinkhorn-Knopp algorithm launched on $(\tilde{A}, \alpha)$.
% \end{proposition}
% \begin{proof}
% The Sinkhorn-Knopp algorithm launched on $(\tilde{A}, \alpha)$ is the following iteration:
% \begin{align*}
%  d_i^{k+1}=\frac{\alpha_i}{\sum_{j \in m+[n]} \tilde{A}_{i,j-m} d^k_{j}}\; ,\; 1\leq i \leq m \\  d_{j}^{k+1}=\frac{\alpha_j}{\sum_{i \in [n]} \tilde{A}_{i,j-m} d^k_{i}}\; ,\; m+1\leq j \leq m+n 
% \end{align*}
% Algorithm~\ref{alg:trunc} launched on $(A,L,U)$ is the following iteration:
% \begin{align*}
%  p_0^{k+1}=\frac{1}{2} \log (\sum_{i \in [m]} \alpha_i e^{p_0^k})
%  -\frac{1}{2} \log(\sum_{j\in m+[n]}\alpha_j e^{-p_0^k})=p_0^k+\frac{1}{2} \log (\sum_{i \in [m]} \alpha_i)-\frac{1}{2} \log(\sum_{j\in m+[n]}\alpha_j)=p_0^k \\
%  p_l^{k+1}=\frac{1}{2} \log( \alpha_l e^{p_l^k})  -\frac{1}{2} \log(\sum_{j\in m+[n]} \tilde{A}_{l,j-m}e^{-p_j^k}) \;, \; 1\leq l \leq m \\
%  p_l^{k+1}=\frac{1}{2} \log(\sum_{i\in [m]} \tilde{A}_{i,l-m}e^{p_i^k}) -\frac{1}{2} \log( \alpha_l e^{-p_l^k})) \;, \; m+1\leq l \leq m+n
% \end{align*}
% \end{proof}

\section{Comparison with alternative algorithms} \label{sec:compalgos}

We give in Table~\ref{tab:compalgos} a comparison of four algorithms for the matrix
balancing problem: interior-reflective Newton method (Matlab fminunc function), coordinate descent, DomEig and ideal HOTS. We considered a small matrix, a medium size matrix and two large matrices. The matrix $A=\begin{bmatrix} \epsilon & 1 \\ 2 & 0 \end{bmatrix}$, with $\epsilon=10^{-3}$, is a nearly imprimitive matrix.
The CMAP matrix is the adjacency matrix of a fragment of the web graph
of size $1,500$. 
The crawl consists of the www.cmap.polytechnique.fr website and surrounding pages.
The NZ Uni matrix comes from a crawl of New Zealand Universities websites, available at~\cite{NZ2006}. It has 413,639 pages. 
The uk2002 matrix comes from a crawl of the .uk name domain with 18,520,486 pages, gathered by UbiCrawler~\cite{BCSU3}. For the matrix balancing problem, we added to the entries of the adjacency matrices
arising from fragments of the web graph
 a small positive constant equal to $1/n$ (to guarantee irreducibility of the matrix).
We launched our numerical experiments on a~personal computer with 4 Intel Xeon CPUs at 2.98~GHz and 8~GB RAM. 

\begin{table}
\centering
 \begin{tabular}{|l|r|r|r|r|}
\hline
  & $A=\begin{bmatrix} \epsilon & 1 \\ 2 & 0 \end{bmatrix}$ & \hspace{-2ex} \begin{tabular}{c} CMAP \\ 1,500 p.\end{tabular} \hspace{-2ex} &  \begin{tabular}{c} NZ Uni \\413,639 p.\end{tabular}  & \hspace{-2.5ex} \begin{tabular}{c} uk2002 \\18,520,486 p.\end{tabular} \hspace{-2.5ex} \\
% \hline
% Schneider's DSS (selection) &  0.0003 s  & 60 s   &   $>600$ s    \\
\hline
$\abs{\lambda_2(P)}$ (Thm.~\ref{thm:convidealHots}) & 0.9993    &  0.8739   &  0.9774   &  0.998   \\
\hline
Matlab's fminunc      &  0.015 s   & 948 s    &  out of mem. & - \\
\hline
DomEig \cite{Johnson-DomEig}               & 1.87 s     & 34.4 s &  $>600$ s  &  -  \\
\hline
Coordinate desc. (Alg. \ref{alg:coorDescentDSS})  &  0.006 s  & 0.03 s   &  6.06 s  &  2,391 s  \\
\hline
Ideal HOTS (Alg. \ref{alg:hots})      &  2.0 s     & 0.02 s &  7.52 s  & 1,868 s \\
\hline
 \end{tabular}
\caption{Execution times for 4 algorithms to solve the matrix balancing problem. General purpose algorithms like Newton methods (fminunc)
 are outperformed by coordinate descent and ideal HOTS for this problem.
DomEig~\cite{Johnson-DomEig} does not seem to be very efficient for these sparse problems.
On the other hand, coordinate descent and ideal HOTS have similar performances. We remark however that
coordinate descent has a better behavior for imprimitive matrices but that we have the expression of the 
rate of convergence of ideal HOTS, that we do not have for coordinate descent.
\label{tab:compalgos}
}
\end{table}

Convex optimization solvers, coding algorithms like quasi-Newton, are heavy machineries that can reach quadratic convergence and can handle general problems.
They however need complex algorithms: one should not program them by scratch but use robust public codes. 
Parallel computation is not so trivial and tuning the parameters may be difficult.
Second order methods also need to compute the Hessian matrix, which may be a large dense matrix.

The DomEig algorithm~\cite{Johnson-DomEig} consists of 2 loops. The inner loop is the power method, it has a linear speed of convergence equal to the spectral gap of the matrix considered. Most of the computational cost lies in the power iterations,
since the outer loop is simple. This algorithm is specially designed for matrix balancing and accepts no extension.
It does not seem to be very efficient for sparse problem, when compared to HOTS and coordinate descent. Moreover, 
the precision required for the determination of the principal eigenvalue has a strong impact on the performance of
the algorithm.

Schneider's coordinate descent DSS algorithm for the matrix balancing problem is a
 very efficient and scalable algorithm. It remains efficient for non primitive matrices.

The ideal HOTS algorithm has a linear speed of convergence equal to the spectral gap of $P$. 
On our experiments, it performs well for medium and large size problems but lacks efficiency for
imprimitive matrices. For the fragments of the web graph, it converged in a
 bit more iterations than coordinate descent but each iteration is a bit simpler.
So the execution times are about the same.
Its advantage compared to coordinate descent is that its implementation
only needs left and right matrix-vector products and element-wise operations (division and square root).
Hence, it does not require to compute and store the transpose of the adjacency matrix,
which can be useful for web scale applications.

In Table~\ref{tab:compalgosHOTS}, we compare coordinate descent and the effective HOTS algorithm
for the effective HOTS problem. Both algorithms scale well and have comparable computational costs.
However, we remark that the rate of convergence seems to deteriorate when the size of the problem
increases. It might be necessary to choose smaller values for $\alpha$ for large graphs in order
to compensate this phenomenon. In the next section, we propose an alternative solution consisting in
normalizing the adjacency matrix prior to computing the HOTS score, and hence minimizing a relative entropy
instead of the entropy function.

In Table~\ref{tab:compalgosRAS}, we give the execution times for the HOTS problem 
with bounds \eqref{eqn:hotsbounds}.

\begin{table}
\centering
 \begin{tabular}{|l|r|r|r|r|}
\hline
                      & $A=\begin{bmatrix} \epsilon & 1 \\ 2 & 0 \end{bmatrix}$ & CMAP & NZ Uni & uk2002  \\
\hline
$\abs{\lambda_2(\nabla F)}$ (Prop.~\ref{prop:contract})   &  0.8846   & 0.946    & 0.995   &  0.9994   \\
\hline
Coordinate descent (Prop.~\ref{prop:effHotsCoorDescent})  &  0.0061 s & 0.0613 s & 35.7  s &  3,809 s   \\
\hline
Effective HOTS (Alg.~\ref{alg:effhots})     		  &  0.0217 s & 0.0589 s &  36.5 s & 3,017 s   \\
\hline
PageRank \cite{Brin-Anatomy}      			  & 0.0216 s  & 0.0195 s & 2.9 s   & 270 s   \\
\hline
 \end{tabular}
\caption{Comparison of coordinate descent and HOTS for the effective HOTS problem
with $\alpha=0.9$. Both algorithms seem to perform equally. In particular, unlike PageRank,
the convergence rate deteriorates with the size of the fragment of the web graph considered. 
\label{tab:compalgosHOTS}
}
\end{table}

\begin{table}
\centering
 \begin{tabular}{|l|r|r|r|}
\hline
                      & $A=\begin{bmatrix} \epsilon & 1 \\ 2 & 0 \end{bmatrix}$ & CMAP &   NZ Uni   \\
% \hline
% Sinkhorn-Knopp \cite{SinkhornKnopp}     &   s   &  s    &  s\\
\hline
Inexact coor. descent (Eq. \eqref{eq:dualtruncSchneider} and \cite{Schneider-scaling}) & 0.019 s  & 4.28 s   &  $>2000$ s     \\
\hline
Exact coor. descent (Alg.~\ref{alg:HotsBounds})   &  0.005 s  & 0.23 s   &  100 s     \\
\hline
 \end{tabular}
\caption{Execution times for two algorithms to solve the effective HOTS problem where some arcs have prescribed
bounds on their flow~\cite{Tomlin-HOTS} (these bounds are determined at random). Exact coordinate descent  (Algorithm~\ref{alg:HotsBounds}) is
faster than inexact coordinate descent. Indeed we do not need any line search for Algorithm~\ref{alg:HotsBounds}.
By comparison with Table~\ref{tab:compalgosHOTS}, we can see that the computational expense is multiplied by less than 4 with the bound constraints.
\label{tab:compalgosRAS}
}
\end{table}

\section{A normalized HOTS algorithm for a better convergence rate}
\label{sec:nomHots}

The previous study has shown two drawbacks of Tomlin's HOTS algorithm.
First of all, its convergence rate seems to deteriorate when the size of the
problem increases. The second problem concerns manipulations of the HOTS score:
when one single page is considered, a very good strategy is to point to no page, and
thus make this page a dangling node in the web graph. This comes from the relation
of HOTS with the anti-Perron ranking (Remark~\ref{rem:antiperron}).
Indeed, the anti-Perron ranking penalizes bad quality links but also 
adding any outlink, even a good quality one, diminishes the score of the page where it has been added.

We now propose a modification of the HOTS algorithm that tackles those two issues.
In order to stop penalizing the presence of hyperlinks on a page, 
we normalize the adjacency matrix, and thus state the entropy optimization problem
as a relative entropy optimization problem. Then the Perron ranking reduces to PageRank and
the normalized anti-Perron ranking does not penalize the number of links any more.

We also need to address the problem of dangling nodes, in which the
normalization of the corresponding row of the adjacency matrix is not defined,
and the reducibility of the adjacency matrix.
A possibility is to keep on inspiring from the PageRank and consider the Google
matrix instead of the normalized adjacency matrix. With this choice,
dangling pages are considered to point to every page, which implies that
they have a low rank in the normalized HOTS score, when compared to the rank given by PageRank.
Instead, we suggest to set Tomlin's effective network model to solve the reducibility
problem and add another fictitious node that point to every page and is pointed to
by every dangling page.

We end up with the following network flow problem where the $(n+2)\times (n+2)$ matrix $M$ is defined by:
\[
M_{i,j}= \begin{cases}
\frac{A_{i,j}}{\sum_k A_{i,k}}  & \text{if } i,j\leq n, \sum_k A_{i,k} \geq 1 \\
0 				& \text{if } i,j\leq n, \sum_k A_{i,k} =0 
\end{cases}
\]
\vspace{-3ex}
\[ M=\begin{bmatrix}  M_{[n],[n]} & f & 1 \\ 1^T & 0 & 1 \\ 1^T & 1 & 0 \end{bmatrix} \]
and $f$ is the 0-1 vector such that $f_i=1$ if and only if $i$ is a dangling node.

\begin{align} \label{eq:normHots}
\max_{\rho \geq 0}  \; - \!\!\!\!\! \sum_{ i,j \in [n+2]} \rho_{i,j} \big(\log(\frac{\rho_{i,j}}{M_{i,j}})-1 \big)
\end{align}
\vspace{-3ex}
\[
\sum_{i,j \in [n+2]} \rho_{ij} =  1  \;, \qquad \sum_{j \in [n+2]} \rho_{i,j} = \sum_{j \in [n+2]} \rho_{j,i} \; , \; \forall i \in [n+2] 
\]
\vspace{-3ex}
\[
\sum_{j \in [n]} \rho_{n+2,j} =  1-\alpha= \sum_{i \in [n]} \rho_{i,n+2} 
\]

We call this optimization problem the normalized HOTS problem.
The normalized HOTS algorithm is defined in the same way as the effective HOTS algorithm but with $M$ instead of $A'$.
As in Lemma~\ref{lem:lambda}, we define
$\mu'= \log(\frac{2\alpha-1}{\sum_{i,j \in [n]} M_{i,j}e^{p_i-p_j}})$,
$a'=\log(\frac{1-\alpha}{2\alpha-1}\frac{\sum_{i,j\in [n]} M_{i,j}e^{p_i-p_j}}{\sum_{j\in[n]}e^{p_{n+1}-p_j}})$,
$b'=-\log(\frac{1-\alpha}{2\alpha-1}\frac{\sum_{i,j\in [n]} M_{i,j}e^{p_i-p_j}}{\sum_{i\in[n]}e^{p_i-p_{n+1}}})$
and the triple $\lambda'(p)=(\mu',a',b')$. 
For $\lambda' \in \mathbb{R}^3$, we denote
\begin{equation*}
 {g'}^{\lambda'}_i(y)=\left ( \frac{\sum_{j\in [n]} M_{j,i} y_j+e^{a'}y_{n+1}}{\sum_{k\in [n]}M_{i,k} (y_k)^{-1} +e^{-b'}(y_{n+1})^{-1}} \right)^{\frac{1}{2}} \enspace.
\end{equation*}
\begin{algorithm}
 \caption{Normalized HOTS algorithm}
\label{alg:normhots}
Given an initial point $y^0 \in \mathbb{R}^n$, the effective HOTS algorithm is defined by
\[
 y^{k+1}={g'}^{\lambda'(\log(y^k))}(y^k)
\]
\end{algorithm}

\begin{proposition}
If there exists a primal feasible point to~\eqref{eq:normHots} with the same pattern as~$A'$,
 then the normalized HOTS algorithm converges with a linear rate of convergence.
\end{proposition}
\begin{proof}
In the proof of Theorem~\ref{thm:convHots} we did not use the fact that the adjacency matrix
$A$ is a 0-1 matrix, only that its elements are nonnegative. Hence, the convergence proof
of the effective HOTS algorithm directly applies to the normalized HOTS algorithm.
\end{proof}

We shall see in Table~\ref{tab:compalgosHOTSNorm} that the convergence properties of
this algorithm seem to be better that those of the classical HOTS algorithm.
Nevertheless, these experimental results are still to be validated by other 
theoretical and numerical studies.

\begin{table}
\centering
 \begin{tabular}{|l|r|r|r|}
\hline
                     % & $A=\begin{bmatrix} \epsilon & 1 \\ 2 & 0 \end{bmatrix}$ 
& CMAP & NZ Uni & uk2002  \\
 \hline
% $\alpha$						& 0.9167   & 0.7827   & 0.7091	& 0.6666 \\
% \hline
% $\abs{\lambda_2(\nabla F)}$ HOTS			& 0.9055   & 0.8565   & 0.9731 & 0.9822 \\
% \hline
% Effective HOTS						& 0.0261   & 0.0298 s & 8.05 s & 2,689 s \\
% \hline
% $\abs{\lambda_2(\nabla F)}$ normHOTS			&  0.1497  &  0.8299  & 0.9044	& 0.9611    \\
% \hline
% Normalized HOTS    			  		&  0.0066  & 0.0461 s & 6.35 s	& 687 s   \\
% \hline
% $\alpha$			&  0.9   & 0.9	& 0.9 \\
% \hline
$\abs{\lambda_2(\nabla F)}$ (classical HOTS)   & 0.946    & 0.995   &  0.9994  \\
\hline
$\abs{\lambda_2(\nabla F)}$ (normalized HOTS)		&   0.906  & 0.988       & 0.96    \\
\hline
Execution times for Normalized HOTS	  		&   0.055 s &  46.24 s	& 752 s   \\
\hline
% PageRank \cite{Brin-Anatomy}  		    		&  0.0195 s & 2.90 s  & 270 s   \\
% \hline
 \end{tabular}
% \caption{Modulation of $\alpha$ for larger graphs and Normalized HOTS 1 (adding diagonal + normalizing $A'+I$), $\alpha=\frac{\tau+1}{2\tau+1}$, $\tau$ = 0.1 x fraction of pages with hyperlinks + fraction of pages without hyperlinks.
\caption{Performances of the normalized HOTS algorithm presented in Section~\ref{sec:nomHots}.
The correlation between the deterioration of the convergence rate of the algorithm, given by the second eigenvalue of the matrix $\nabla F$, and the size of the data set does not hold any more.
Moreover, on all the tests we performed, the convergence rate remained under 0.99.
\label{tab:compalgosHOTSNorm}
}
\end{table}

\section*{Acknowledgment}

 I gratefully thank St\'ephane Gaubert who advised me to apply nonlinear Perron-Frobenius theory
to tackle the convergence of HOTS algorithm.

\ifthenelse{\boolean{arxiv}}
{
\bibliographystyle{alpha}
}
{
\bibliographystyle{siam.bst}
}
\bibliography{hots,../PAPERS2011/perronOpt,../PAPEROPTIMRANK/pagerank}

\end{document}